\documentclass[twoside]{univzil}
\usepackage[english]{babel}
\usepackage[T1]{fontenc}
\usepackage[utf8]{inputenc} 
\usepackage{amsmath}
\usepackage{amscd,amssymb,amsfonts}
\usepackage{graphicx} 
\usepackage{url}

\begin{document}

\setcounter{page}{1}                   
\renewcommand\volumeZA{2014}          
\def\emailh#1{\email{\href{mailto:#1}{#1}}}
\def\urladdressh#1{\urladdress{\url{#1}}}

\title[On the hyperbolic triangle centers]{On the hyperbolic triangle centers}
\author[\'A. G.Horv\'ath]{\'Akos G.Horv\'ath}
\date{25 Aug, 2014}

\address{\'A. G.Horv\'ath, Dept. of Geometry, Budapest University of Technology,
Egry J\'ozsef u. 1., Budapest, Hungary, 1111}
\email{ghorvath@math.bme.hu}

\keywords{cycle, hyperbolic plane, triangle centers}
\subjclass{51M10, 51M15}{}

\begin{abstract}
Using the method of C. V\"or\"os, we establish results on hyperbolic plane geometry, related to triangles.
In this note we investigate the orthocenter, the concept of isogonal conjugate and some further center as of the symmedian of a triangle. We also investigate the role of the "Euler line" and the pseudo-centers of a triangle.
\end{abstract}

\newtheorem{theorem}{Theorem}[section]
\newtheorem{corollary}[theorem]{Corollary}
\newtheorem{lemma}[theorem]{Lemma}
\newtheorem{exmple}[theorem]{Example}
\newtheorem{defn}[theorem]{Definition}
\newtheorem{proposition}[theorem]{Proposition}
\newtheorem{conjecture}[theorem]{Conjecture}
\newtheorem{rmrk}[theorem]{Remark}
\newenvironment{definition}{\begin{defn}\normalfont}{\end{defn}}
\newenvironment{remark}{\begin{rmrk}\normalfont}{\end{rmrk}}
\newenvironment{example}{\begin{exmple}\normalfont}{\end{exmple}}
\newtheorem*{remarque}{Remark}
\newcommand\bbreak{\allowdisplaybreaks}
\newcommand\dd{\mathop{\rm d\!}\nolimits}
\newcommand\sgn{\mathop{\rm sgn}\nolimits}

\maketitle\section{Introduction and preliminaries}

In an earlier work of the author (\cite{gho 2},\cite{gho 3}) investigated the concept of distance extracted from the work of Cyrill V\"or\"os and, translating the standard methods of Euclidean plane geometry into the hyperbolic plane, apply it for various configurations. We gave a model independent construction for the famous problem of Malfatti (discussed in \cite{gho 1}) and gave some interesting formulas connected with the geometry of hyperbolic triangles. In this paper we follow the investigations above for some other concept of the hyperbolic triangle.

\subsection{Well-known formulas on hyperbolic trigonometry}

The points $A,B,C$ denote the vertices of a triangle. The lengths of the edges opposite to these vertices are $a,b,c$, respectively. The angles at $A,B,C$ are denoted by $\alpha, \beta, \gamma$, respectively. If the triangle has a right angle, it is always at $C$. The symbol $\delta$ denotes half of the area of the triangle; more precisely, we have $2\delta=\pi-(\alpha+\beta+\gamma)$.
\begin{itemize}
\item {\bf Connections between the trigonometric and hyperbolic trigonometric functions:}
$$
\sinh a=\frac{1}{i}\sin (ia), \quad \cosh a=\cos (ia), \quad \tanh a=\frac{1}{i}\tan (ia)
$$
\item {\bf Law of sines:}
\begin{equation}
\sinh a :\sinh b :\sinh c=\sin \alpha :\sin \beta :\sin \gamma
\end{equation}
\item {\bf Law of cosines:}
\begin{equation}
\cosh c=\cosh a\cosh b-\sinh a\sinh b \cos \gamma
\end{equation}
\item {\bf Law of cosines on the angles:}
\begin{equation}
\cos \gamma=-\cos\alpha\cos\beta+\sin\alpha\sin\beta \cosh c
\end{equation}
\item {\bf The area of the triangle:}
\begin{eqnarray}
T:=2\delta& = &\pi-(\alpha+\beta+\gamma)\\
\nonumber \tan \frac{T}{2}& = & \left(\tanh \frac{a_1}{2}+\tanh \frac{a_1}{2}\right)\tanh \frac{m_a}{2}
\end{eqnarray}
where $m_a$ is the height of the triangle corresponding to $A$ and $a_1,a_2$ are the signed lengths of the segments into which the foot point of the height divides the side $BC$.
\item {\bf Heron's formula:}
\begin{equation}
\tan \frac{T}{4}=\sqrt{\tanh \frac{s}{2}\tanh \frac{s-a}{2}\tanh \frac{s-b}{2}\tanh \frac{s-c}{2}}
\end{equation}
\item{\bf Formulas on Lambert's quadrangle:} The vertices of the quadrangle are $A,B,C,D$ and the lengths of the edges are $AB=a,BC=b,CD=c$ and $DA=d$, respectively. The only angle which is not right-angle is $ BCD\measuredangle=\varphi$. Then, for the sides, we have:
$$
\tanh b=\tanh d\cosh a, \quad \tanh c=\tanh a\cosh d,
$$
and
$$
 \sinh b=\sinh d\cosh c, \quad \sinh c=\sinh a\cosh b ,
$$
moreover, for the angles, we have:
$$
\cos \varphi=\tanh b\tanh c=\sinh a\sinh d, \quad \sin \varphi=\frac{\cosh d}{\cosh b}=\frac{\cosh a}{\cosh c},
$$
and
$$
\tan \varphi =\frac{1}{\tanh a\sinh b}=\frac{1}{\tanh d\sinh c}.
$$
\end{itemize}

\subsection{The distance of the points and the lengths of the segments}

In \cite{gho 2} we extracted the concepts of the distance of real points following the method of the book of Cyrill V\"or\"os (\cite{voros}). We extend the plane with two types of points, one type of the points at infinity and the other one the type of ideal points. In a projective model these are the boundary and external points of a model with respect to the embedding real projective plane. Two parallel lines determine a point at infinity and two ultraparallel lines an ideal point which is the pole of their common transversal. Now the concept of the line can be extended; a line is real if it has real points (in this case it also has two points at infinity and the other points on it are ideal points being the poles of the real lines orthogonal to the mentioned one). The extended real line is a closed compact set with finite length. We also distinguish the line at infinity which contains precisely one point at infinity and the so-called ideal line which contains only ideal points. By definition the common lengths of these lines are $\pi ki$, where $k$ is a constant of the hyperbolic plane and $i$ is the imaginary unit. In this paper we assume that $k=1$. Two points on a line determine two segments $AB$ and $BA$.
The sum of the lengths of these segments is $AB+BA=\pi i$. We define the length of a segment as an element of the linearly ordered set $\bar{\mathbb{C}}:=\overline{\mathbb{R}}+ \mathbb{R}\cdot i$. Here $\overline{\mathbb{R}}=\mathbb{R}\cup\{\pm \infty\}$ is the linearly ordered set of real numbers extracted with two new numbers with the "real infinity" $\infty $ and its additive inverse $-\infty$. The infinities  can be considered as new "numbers" having the properties that either "there is no real number greater or equal to $\infty$" or "there is no real number less or equal to $-\infty$". We also introduce the following operational rules: $\infty+\infty=\infty$, $-\infty+(-\infty)=-\infty$, $\infty+(-\infty)=0$ and  $\pm\infty+a=\pm\infty$ for real $a$. It is obvious that $\overline{\mathbb{R}}$ is not a group, the rule of associativity holds only such expressions which contain at most two new objects. In fact, $0=\infty+(-\infty)=(\infty+\infty)+(-\infty)=\infty+(\infty+(-\infty))=\infty$ is a contradiction. We also require that the equality $\pm\infty+b i=\pm\infty +0 i$ holds for every real number $b$ and for brevity we introduce the respective notations $\infty:=\infty +0 i$ and $-\infty:=-\infty +0 i$. We extract the usual definition of hyperbolic function based on the complex exponential function by the following formulas
$$
\cosh (\pm \infty):=\infty, \sinh (\pm \infty):=\pm \infty, \mbox{ and } \tanh(\pm \infty):=\pm 1.
$$
We also assume that $\infty \cdot \infty=(-\infty)\cdot(-\infty)=\infty$, $\infty\cdot(-\infty)=-\infty$ and $\alpha \cdot (\pm \infty)=\pm \infty$.

Assuming that the trigonometric formulas of hyperbolic triangles are also valid with ideal vertices the definition of the mentioned lengths of the complementary segments of a line are given.
We defined all of the possible lengths of a segment on the basis of the type of the line contains them.

The definitions of the respective cases can be found in Table 1. We abbreviate the words real, infinite and ideal by symbols ${\mathcal R}$, ${\mathcal I}n$ and ${\mathcal I}d$, respectively. $d$ means a real (positive) distance of the corresponding usual real elements which are a real point or the real polar line of an ideal point, respectively. Every box in the table contains two numbers which are the lengths of the two segments determined by the two points. For example, the distance of a real and an ideal point is a complex number. Its real part is the distance of the real point to the polar of the ideal point with a sign, this sign is positive in the case when the polar line intersects the segment between the real and ideal points, and is negative otherwise. The imaginary part of the length is $(\pi/2)i$, implying that the sum of the lengths of two complementary segments of this projective line has total length $\pi i$.
Consider now an infinite point. This point can also be considered as the limit of real points or limit of ideal points of this line. By definition the distance from a point at infinity of a real line to any other real or infinite point of this line is $\pm \infty$ according to that it contains or not ideal points. If, for instance, $A$ is an infinite point and $B$ is a real one, then the segment $AB$ contains only real points has length $\infty$. It is clear that with respect to the segments on a real line the length-function is continuous.

\begin{table}
\begin{tabular}{cc|c|c|c|}
  \cline{3-5}
  & & \multicolumn{3}{|c|}{$B$} \\ \cline{2-5}
  &\multicolumn{1}{|c|}{ } &  ${\mathcal R}$  & ${\mathcal I}n$  & ${\mathcal I}d$  \\ \cline{1-5}
 \multicolumn{1}{|c|}{}& \multicolumn{1}{|c|}{${\mathcal R}$} & \begin{tabular}{c}  $AB=d$ \\ $BA=-d+\pi i$  \end{tabular}   & \begin{tabular}{c} $AB=\infty $ \\$BA=-\infty$ \end{tabular} & \begin{tabular}{c} $AB=d+\frac{\pi}{2} i$ \\  $BA=-d+\frac{\pi}{2} i$ \end{tabular} \\ \cline{2-5}
 \multicolumn{1}{|c|}{$A$}& \multicolumn{1}{|c|}{${\mathcal I}n$} &   & \begin{tabular}{c} $AB=\infty $ \\ $BA=-\infty$ \end{tabular}  & \begin{tabular}{c} $AB=\infty$ \\$BA=-\infty $ \end{tabular} \\  \cline{2-5}
 \multicolumn{1}{|c|}{}& \multicolumn{1}{|c|}{${\mathcal I}d$} & &  & \begin{tabular}{c}  $AB=d+\pi i$ \\ $BA=-d$ \end{tabular} \\
  \hline
\end{tabular}
\vspace{0,2cm}
\caption{Distances on the real line}
\end{table}

We can check that the length of a segment for which either $A$ or $B$ is an infinite point is indeterminable. On the other hand if we consider the polar of the ideal point $A$ we get a real line through $B$. The length of a segment connecting the (ideal) point $A$ and one of the points of its polar is $(\pi/2)i$. This means that we can define the length of a segment between $A$ and $B$ also as this common value. Now if we want to preserve the additivity property of the lengths of segments on a line at infinity, too then we must give the pair of values $0, \pi i$ for the lengths of segment with ideal ends.
    The Table 2 collects these definitions.

\begin{table}
\begin{tabular}{cc|c|c|}
  \cline{3-4}
 & & \multicolumn{2}{|c|}{$B$} \\ \cline{2-4}
      & \multicolumn{1}{|c|}{} & ${\mathcal I}n$  & ${\mathcal I}d$  \\ \cline{1-4}
  \multicolumn{1}{|c|}{$A$} & \multicolumn{1}{|c|}{${\mathcal I}n$}  & \begin{tabular}{c} $AB=0 $ \\ $BA=\pi i$ \end{tabular}  & \begin{tabular}{c} $AB=\frac{\pi}{2}i$ \\ $BA=\frac{\pi}{2}i $ \end{tabular} \\  \cline{2-4}
    \multicolumn{1}{|c|}{} & \multicolumn{1}{|c|}{${\mathcal I}d$} & & \begin{tabular}{c}  $AB=0$ \\ $BA=\pi i$ \end{tabular} \\
  \hline
\end{tabular}
\vspace{0.2cm}
\caption{Distances on the line at infinity}
\end{table}

\begin{figure}[htbp]
\centerline{\includegraphics[height=4cm]{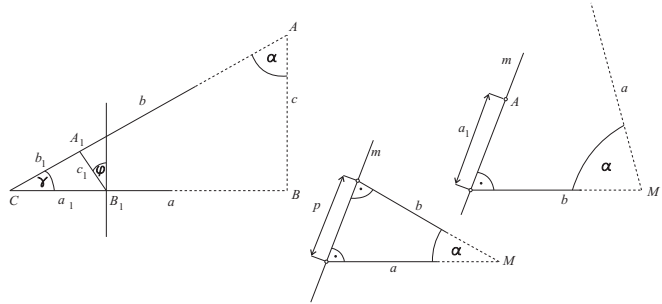}}
\caption{The cases of the ideal segment and angles}
\end{figure}

The last situation contains only one case: $A$, $B$ and $AB$ are ideal elements, respectively. As we showed in \cite{gho 2}
\emph{The length of an ideal segment on an ideal line is the angle of their polars multiplied by the imaginary unit $i$.}

\begin{table}[htbp]
\begin{tabular}{cc|c|c|c|}
  \cline{3-5} & & \multicolumn{3}{|c|}{$a$} \\ \cline{3-5}
  & & ${\mathcal R}$ & ${\mathcal I}n$ & ${\mathcal I}d$ \\ \hline
  \multicolumn{1}{|c|}{}& \multicolumn{1}{|c|}{${\mathcal R}$} & \begin{tabular}{|c|c|c|} \multicolumn{3}{|c|}{$M$} \\ \hline ${\mathcal R}$ & ${\mathcal I}n$ & ${\mathcal I}d$ \\ \hline
                                    \begin{tabular}{c} $\varphi$ \\
                                    $\pi-\varphi$
                                    \end{tabular} & \begin{tabular}{c} $0$ \\
                                    $\pi$
                                    \end{tabular}  & \begin{tabular}{c} $\frac{p}{i}$ \\
                                    $\pi -\frac{p}{i}$
                                    \end{tabular} \\
                                    \end{tabular} &
                                                    \begin{tabular}{|c|c|} \multicolumn{2}{|c|}{$M$} \\ \hline  ${\mathcal I}n$ & ${\mathcal Id}$ \\ \hline
                                                     \begin{tabular}{c} $\frac{\pi}{2}$ \\
                                                                         $\frac{\pi}{2}$
                                                     \end{tabular}  &
                                                                         \begin{tabular}{c} $\infty$ \\
                                                                                            $-\infty$
                                                                                          \end{tabular} \\
                                                                                           \end{tabular} &
                                                                                                           \begin{tabular}{|c|}  \multicolumn{1}{|c|}{$M$} \\ \hline ${\mathcal I}d$ \\ \hline
                                                                                                           \begin{tabular}{c} $\frac{\pi}{2}+\frac{a_1}{i}$ \\
                                                                                                            $\frac{\pi}{2}-\frac{a_1}{i}$
                                                                                                            \end{tabular} \\
                                                                                                             \end{tabular}\\ \cline{2-5}
   \multicolumn{1}{|c|}{$b$} & \multicolumn{1}{|c|}{${\mathcal I}n$} & &\begin{tabular}{|c|}
                                \multicolumn{1}{|c|}{$M$} \\ \hline ${\mathcal I}d$ \\ \hline
                                    $\infty$ \\
                                    $-\infty$
                                    \end{tabular}
                                                 &\begin{tabular}{|c|} \multicolumn{1}{|c|}{$M$} \\ \hline
                                                   ${\mathcal I}d$ \\ \hline
                                                    $\infty$ \\
                                                     $-\infty$
                                                   \end{tabular}  \\ \cline{2-5}
  \multicolumn{1}{|c|}{} &\multicolumn{1}{|c|}{${\mathcal I}d$} &  &  & \begin{tabular}{|c|} \multicolumn{1}{|c|}{$M$} \\ \hline
                                    ${\mathcal I}d$ \\ \hline
                                    $\frac{p}{i}$ \\
                                    $\pi-\frac{p}{i}$
                                    \end{tabular} \\
  \hline
\end{tabular}
\caption{Angles of lines}
\end{table}
Similarly as in the previous paragraph we can deduce the angle between arbitrary kind of lines (see Table 3). In Table 3, $a$ and $b$ are the given lines, $M=a\cap b$ is their intersection point, $m$ is the polar of $M$ and $A$ and $B$ is the poles of $a$ and $b$, respectively. The numbers $p$ and $a_1$ represent real distances, can be seen on Fig. 1, respectively.
The general connection between the angles and distances is the following:
\emph{ Every distance of a pair of points is the measure of the angle of their polars multiplied by $i$. The domain of the angle can be chosen on such a way, that we are going through the segment by a moving point and look at the domain which described by the moving polar of this point. }

\subsection{Results on the three mean centers}

There are many interesting statements on triangle centers. In this section we mention some of them concentrating only the centroid, circumcenters and incenters, respectively
 (\cite{gho 2},\cite{gho 3}).

The notation of this subsection follows the previous part of this paper: the vertices of the triangle are $A,B,C$, the corresponding angles are $\alpha,\beta,\gamma$ and the lengths of the sides opposite to the vertices are $a,b,c$, respectively. We also use the notion $2s=a+b+c$ for the perimeter of the triangle. Let denote $R,r,r_A,r_B,r_C$ the radius of the circumscribed cycle, the radius of the inscribed cycle (shortly incycle), and the radiuses of the escribed cycles opposite to the vertices $A,B,C$, respectively. We do not assume that the points $A,B,C$ are real and the distances are positive numbers. In most cases the formulas are valid for ideal elements and elements at infinity and when the distances are complex numbers, respectively. The only exception when the operation which needs to the examined formula is understandable. Before the examination of hyperbolic triangle centers we collect some further important formulas on hyperbolic triangles. We can consider them in our extracted manner.
\subsubsection{Staudtian and angular Staudtian of a hyperbolic triangle:}
The \emph{Staudtian of a hyperbolic triangle} something-like similar (but definitely distinct) to the concept of the Euclidean area. In spherical trigonometry the twice of this very important quantity called by Staudt the sine of the trihedral angle $O-ABC$ and later Neuberg suggested the names (first) ``Staudtian'' and the ``Norm of the sides'', respectively. We prefer in this paper the name ``Staudtian'' as a token of our respect for the great geometer Staudt.
Let
$$
n=n(ABC):=\sqrt{\sinh s\sinh (s-a)\sinh (s-b)\sinh (s-c)},
$$ then we have
\begin{equation}
\sin \frac{\alpha}{2}\sin \frac{\beta}{2}\sin \frac{\gamma}{2}=\frac{n^2}{\sinh s\sinh a\sinh b\sinh c}.
\end{equation}
This observation leads to the following formulas on the Staudtian:
\begin{equation}
\sin \alpha =\frac{2n}{\sinh b\sinh c}, \quad \sin \beta =\frac{2n}{\sinh a\sinh c}, \quad \sin \gamma =\frac{2n}{\sinh a\sinh b}. \end{equation}
From the first equality of (4.2) we get that
\begin{equation}
n=\frac{1}{2} \sin \alpha \sinh b\sinh c= \frac{1}{2}\sinh h_C\sinh c,
\end{equation}
where $h_C$ is the height of the triangle corresponding to the vertex $C$.
As a consequence of this concept we can give homogeneous coordinates for the points of the plane with respect to a basic triangle as follows:
\begin{defn}
Let $ABC$ be a non-degenerated reference triangle of the hyperbolic plane. If $X$ is an arbitrary point we define its coordinates by the ratio of the Staudtian $X:=\left(n_A(X):n_B(X):n_C(X)\right)$
where $n_A(X)$, $n_B(X)$ and $n_C(X)$ means the Staudtian of the triangle $XBC$, $XCA$ and $XAB$, respectively. This triple of coordinates is the \emph{triangular coordinates} of the point $X$ with respect to the triangle $ABC$.
\end{defn}
Consider finally the ratio of section $(BX_AC)$ where $X_A$ is the foot of the transversal $AX$ on the line $BC$. If $n(BX_AA)$, $n(CX_AA)$   mean the Staudtian of the triangles $BX_AA$, $CX_AA$, respectively then using (4.3) we have
$$
(BX_AC)=\frac{\sinh BX_A}{\sinh X_AC}=\frac{\frac{1}{2}\sinh h_C\sinh BX_A}{\frac{1}{2}\sinh h_C\sinh X_A C}=\frac{n(BX_AA)}{n(CX_AA)}=
$$
$$
=\frac{\frac{1}{2}\sinh c\sinh AX_A\sin(BAX_A)\measuredangle}{\frac{1}{2}\sinh b\sinh AX_A\sin(CAX_A)\measuredangle}=\frac{\sinh c\sinh AX\sin(BAX_A)\measuredangle}{\sinh b\sinh AX\sin(CAX_A)\measuredangle}=\frac{n_C(X)}{n_B(X)},
$$
proving that
\begin{equation}
(BX_AC)=\frac{n_C(X)}{n_B(X)}, (CX_BA)=\frac{n_A(X)}{n_C(X)}, (AX_CB)=\frac{n_B(X)}{n_A(X)}.
\end{equation}

The \emph{angular Staudtian} of the triangle defined by the equality:
$$
N=N(ABC):=\sqrt{\sin\delta\sin(\delta+\alpha)\sin(\delta+\beta)\sin(\delta+\gamma)},
$$
is the "dual" of the concept of Staudtian and thus we have similar formulas on it. From the law of cosines on the angles we have
$\cos \gamma=-\cos \alpha\cos \beta + \sin \alpha\sin\beta \cosh c$
and adding to this the addition formula of the cosine function we get that
$$
\sin \alpha\sin\beta(\cosh c-1)=\cos \gamma +\cos (\alpha +\beta)=2\cos\frac{\alpha+\beta+\gamma}{2}\cos\frac{\alpha+\beta-\gamma}{2}.
$$
From this we get that
\begin{equation}
\sinh\frac{c}{2}=\sqrt{\frac{\sin{\delta}\sin{(\delta+\gamma})}{\sin \alpha\sin\beta}}.
\end{equation}
Analogously we get that
\begin{equation}
\cosh \frac{c}{2}=\sqrt{\frac{\sin{(\delta+\beta)}\sin{(\delta+\alpha})}{\sin \alpha\sin\beta}}.
\end{equation}
From these equations
\begin{equation}
\cosh \frac{a}{2}\cosh \frac{b}{2}\cosh \frac{c}{2}=\frac{N^2}{\sin \alpha\sin\beta\sin \gamma\sin\delta}.
\end{equation}
Finally we also have that
\begin{equation} \sinh a =\frac{2N}{\sin \beta\sin \gamma}, \quad \sinh b=\frac{2N}{\sin \alpha\sin \gamma}, \quad \sinh c =\frac{2N}{\sin \alpha\sin \beta}, \end{equation}
and from the first equality of (4.8) we get that
\begin{equation}
N=\frac{1}{2} \sinh a \sin\beta\sin\gamma= \frac{1}{2}\sinh h_C\sin\gamma.
\end{equation}
The connection between the two Staudtians gives by the formula
\begin{equation}
2n^2=N\sinh a\sinh b\sinh c.
\end{equation}
In fact, from (7) and (13) we get that
$$
\sin\alpha\sinh a=\frac{4nN}{\sin \beta\sin\gamma\sinh b\sinh c}
$$
implying that
$$
\sin\alpha\sin \beta\sin\gamma\sinh a\sinh b\sinh c=4nN.
$$
On the other hand from (7) we get immediately that
$$
\sin\alpha\sin\beta\sin\gamma=\frac{8n^3}{\sinh^2a\sinh^2b\sinh^2c}
$$
and thus
$$
2n^2=\sinh a\sinh b\sinh c N,
$$
as we stated. The connection between the two types of the Staudtian can be understood if we dived to the first equality of (7) by the analogous one in (19) we get that $\frac{\sin \alpha}{\sinh a}=\frac{n}{N}\frac{\sin \beta}{\sinh b}\frac{\sin \gamma}{\sinh c}$
implying the equality
\begin{equation}
\frac{N}{n}=\frac{\sin \alpha}{\sinh a}.
\end{equation}

\subsubsection{On the centroid (or median point) of a triangle}
We denote the medians of the triangle by $AM_A,BM_B$ and $CM_C$, respectively. The feet of the medians $M_A$,$M_B$ and $M_C$. The existence of their common point $M$ follows from the Menelaos-theorem (\cite{szasz}). For instance if $AB$, $BC$ and $AC$ are real lines and the points $A,B$ and $C$ are ideal points then we have that $AM_C=M_CB=d=a/2$ implies that $M_C$ is the middle point of the real segment lying on the line $AB$ between the intersection points of the polars of $A$ and $B$ with $AB$, respectively (see Fig. 2).

\begin{figure}[htbp]
\centerline{\includegraphics[scale=0.6]{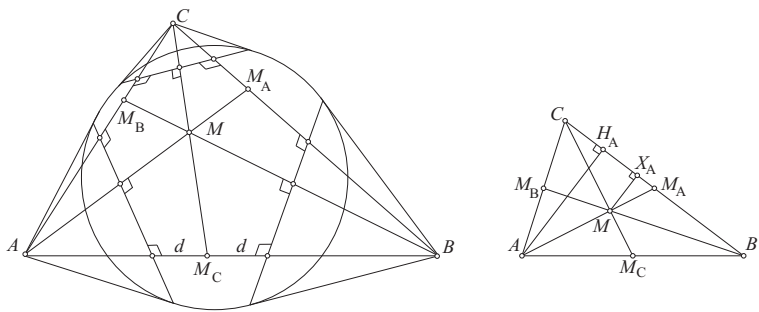}}
\caption{Centroid of a triangle with ideal vertices.}
\end{figure}

\begin{theorem}[\cite{gho 3}] We have the following formulas connected with the centroid:
\begin{equation}
n_A(M)=n_B(M)=n_C(M)
\end{equation}
\begin{equation}
\frac{\sinh AM}{\sinh MM_A}=2 \cosh \frac{a}{2}
\end{equation}
\begin{equation}
\frac{\sinh AM_A}{\sinh MM_A}=\frac{\sinh BM_B}{\sinh MM_B}=\frac{\sinh CM_C}{\sinh MM_C}=\frac{n}{n_A(M)}.
\end{equation}
\begin{equation}
\sinh d'_M=\frac{\sinh d'_A +\sinh d'_B +\sinh d'_C}{\sqrt{1+2(1+\cosh a +\cosh b+\cosh c)}}.
\end{equation}
where $d'_A$, $d'_B$, $d'_C$, $d'_M$ mean the signed distances of the points $A,B,C,M$ to a line $y$, respectively.
Finally we have
\begin{equation}
\cosh YM=\frac{\cosh YA +\cosh YB +\cosh YC}{\frac{n}{n_A(M)}}.
\end{equation}
where $Y$ a point of the plane. (20) and (21) are called the ``center of gravity'' property of $M$ and the ``minimality'' property of $M$, respectively.
\end{theorem}

\subsubsection{On the center of the circumscribed cycle}

Denote by $O$ the center of the circumscribed cycle of the triangle $ABC$. In the extracted plane $O$ always exists and could be a real point, point at infinity or ideal point, respectively. Since we have two possibilities to choose the segments $AB$, $BC$ and $AC$ on their respective lines, we also have four possibilities to get a circumscribed cycle. One of them corresponds to the segments with real lengths and the others can be gotten if we choose one segment with real length and two segments with complex lengths, respectively. If $A,B,C$ are real points the first cycle could be circle, paracycle or hypercycle, but the other three are always hypercycles, respectively. For example, let $a'=a=BC$ is a real length and  $b'=-b+\pi i$, $c'=-c +\pi i$ are complex lengths, respectively. Then we denote by $O_A$ the corresponding (ideal) center and by $R_A$ the corresponding (complex) radius. We also note that the latter three hypercycle have geometric meaning. These are those hypercycles which fundamental lines contain a pair from the midpoints of the edge-segments and contain that vertex of the triangle which is the meeting point of the corresponding edges.
\begin{theorem} The following formulas are valid on the circumradiuses:
\begin{equation} \tanh R=\frac{\sin\delta}{N}, \quad \tanh R_A=\frac{\sin(\delta+\alpha)}{N} \end{equation}
\begin{equation} \tanh R=\frac{2\sinh\frac{a}{2}\sinh\frac{b}{2}\sinh\frac{c}{2}}{n} \quad
\tanh R_A=\frac{2\sinh\frac{a}{2}\cosh\frac{b}{2}\cosh\frac{c}{2}}{n}
\end{equation}
\begin{equation}
n_A(0):n_B(O)=\cos (\delta+\alpha)\sinh a:\cos (\delta+\beta)\sinh b
\end{equation}
\end{theorem}

\subsubsection{On the center of the inscribed and escribed cycles}

We are aware that the bisectors of the interior angles of a hyperbolic triangle are concurrent at a point $I$, called the incenter,  which is equidistant from the sides of the triangle. The radius of the \emph{incircle} or \emph{inscribed circle}, whose center is at the incenter and touches the sides, shall be designated by $r$. Similarly the bisector of any interior angle and those of the exterior angles at the other vertices, are concurrent at point outside the triangle; these three points are called \emph{excenters}, and the corresponding tangent cycles \emph{excycles} or \emph{escribed cycles}. The excenter lying on $AI$ is denoted ba $I_A$, and the radius of the escribed cycle with center at $I_A$ is $r_A$. We denote by $X_A$, $X_B$, $X_C$ the points where the interior bisectors meets $BC$, $AC$, $AB$, respectively. Similarly $Y_A$, $Y_B$ and $Y_C$ denote the intersection of the exterior bisector at $A$, $B$ and $C$ with $BC$, $AC$ and $AB$, respectively.
\begin{figure}[htbp]
\centerline{\includegraphics[scale=0.7]{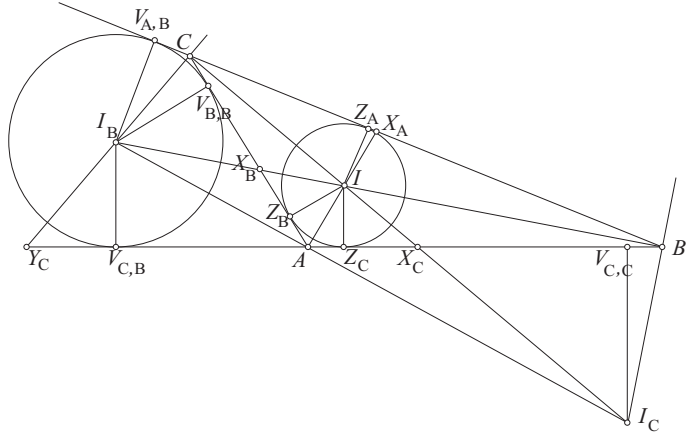}}
\caption{Incircles and excycles.}
\end{figure}
We note that the excenters and the points of intersection of the sides with the bisectors of the corresponding exterior angle could be points at infinity or also could be ideal points. Let denote the touching points of the incircle $Z_A$, $Z_B$ and $Z_C$ on the lines $BC$, $AC$ and $AB$, respectively and the touching points of the excycles with center $I_A$, $I_B$ and $I_C$ are the triples $\{V_{A,A},V_{B,A},V_{C,A}\}$, $\{V_{A,B},V_{B,B},V_{C,B}\}$ and $\{V_{A,C},V_{B,C},V_{C,C}\}$, respectively (see in Fig. 3).
\begin{theorem}[\cite{gho 3}]
On the radiuses $r$, $r_A$, $r_B$ or $r_C$ we have the following formulas .
\begin{equation} \tanh r=\frac{n}{\sinh s}, \quad \tanh r_A=\frac{n}{\sinh (s-a)}
\end{equation}
\begin{equation}
\tanh r=\frac{N}{2\cos \frac{\alpha}{2}\cos \frac{\beta}{2}\cos \frac{\gamma}{2}},
\end{equation}
\begin{eqnarray}
\coth r & = & \frac{\sin(\delta+\alpha)+\sin(\delta+\beta)+\sin(\delta+\gamma)+\sin\delta}{2N} \\
\coth r_A & = & \frac{-\sin(\delta+\alpha)+\sin(\delta+\beta)+\sin(\delta+\gamma)-\sin\delta}{2N}
\end{eqnarray}
\begin{eqnarray}
\tanh R+\tanh R_A & = & \coth r_B+\coth r_C \\
\nonumber
\tanh R_B+\tanh R_C & = & \coth r+\coth r_A \\
\nonumber
\tanh R +\coth r & = & \frac{1}{2}\left(\tanh R+\tanh R_A+\tanh R_B+\tanh R_C\right)
\end{eqnarray}
\begin{eqnarray}
n_A(I):n_B(I):n_C(I) & = &  \sinh a:\sinh b:\sinh c\\
n_A(I_A):n_B(I_A):n_C(I_A) & = &  -\sinh a :\sinh b: \sinh c
\end{eqnarray}
\end{theorem}
The following formulas connect the radiuses of the circles and the lengths of the edges of the triangle.

\begin{theorem}
Let $a,b,c,s,r_A,r_B,r_C,r,R$ be the values defined for a hyperbolic triangle above. Then we have the following formulas:
\begin{equation}
-\coth r_A-\coth r_B-\coth r_C+\coth r=2\tanh R
\end{equation}
\begin{eqnarray}
\lefteqn{\coth r_A\coth r_B + \coth r_A\coth r_C+\coth r_B\coth r_C=} \\
 && \nonumber =\frac{1}{\sinh s\sinh(s-a)}+\frac{1}{\sinh s\sinh(s-b)}+\frac{1}{\sinh s\sinh(s-c)}
\end{eqnarray}
\begin{eqnarray}
\lefteqn{\tanh r_A\tanh r_B +  \tanh r_A\tanh r_C+\tanh r_B\tanh r_C=} \\
&& \nonumber =  \frac{1}{2}\left(\cosh (a+b)+\cosh (a+c)+\cosh (b+c)-\cosh a-\cosh b-\cosh c\right)
\end{eqnarray}
\begin{eqnarray}
\lefteqn{\coth r_A +  \coth r_B+\coth r_C=} \\
&& \nonumber  =\frac{1}{\tanh r}\left(\cosh a+\cosh b+\cosh c-\coth s\left(\sinh a+\sinh b+\sinh c\right)\right)
\end{eqnarray}
\begin{eqnarray}
\lefteqn{\tanh r_A + \tanh r_B+\tanh r_C=}\\
\nonumber && =\frac{1}{2\tanh r}\left(\cosh a+\cosh b+\cosh c-\cosh (b-a)-\cosh (c-a)-\cosh (c-b)\right)
\end{eqnarray}
\begin{eqnarray}
\lefteqn{ 2(\sinh a\sinh b + \sinh a\sinh c +\sinh b\sinh c)=} \\
\nonumber & = &\tanh r\left(\tanh r_A+\tanh r_B+\tanh r_C\right)+ \tanh r_A\tanh r_B+   \\
\nonumber & & { +\tanh r_A\tanh r_C+\tanh r_B\tanh r_C}
\end{eqnarray}
\end{theorem}

\begin{proof}
From (27),(28) and (22) we get that
$$
-\coth r_A-\coth r_B-\coth r_C+\coth r=2\frac{\sin \delta}{N}=2\tanh R,
$$
as we stated in (32).

To prove (33) consider the equalities in (25) from which
$$
\coth r_A\coth r_B+\coth r_A\coth r_C+\coth r_B\coth r_C=
$$
$$
=\frac{\sinh(s-a)\sinh(s-b)+\sinh(s-a)\sinh(s-c)+\sinh(s-c)\sinh(s-b)}{n^2}=
$$
$$
\frac{1}{\sinh s\sinh(s-a)}+\frac{1}{\sinh s\sinh(s-b)}+\frac{1}{\sinh s\sinh(s-c)}
$$
Similarly we also get (34):
$$
\tanh r_A\tanh r_B+\tanh r_A\tanh r_C+\tanh r_B\tanh r_C=
$$
$$
=\sinh s\sinh(s-a)+\sinh s\sinh(s-b)+\sinh s\sinh(s-c)=
$$
$$
=\frac{1}{2}\left(\cosh (a+b)+\cosh (a+c)+\cosh (b+c)-\cosh a-\cosh b-\cosh c\right).
$$
Since we have
$$
-2\tanh R+\coth r=\coth r_A+\coth r_B+\coth r_C=
$$
$$
=\frac{\sinh(s-a)+\sinh(s-b)+\sinh(s-c)}{n}=
$$
$$
=\frac{\left(\sinh(s-a)+\sinh(s-b)+\sinh(s-c)\right)}{\sinh s\tanh r}=
$$
$$
=\frac{\cosh a+\cosh b+\cosh c-\coth s\left(\sinh a+\sinh b+\sinh c\right)}{\tanh r}
$$
(34) is given.
Furthermore we also have
$$
\tanh r_A+\tanh r_B+\tanh r_C=
$$
$$
=\frac{n\left(\sinh(s-a)\sinh(s-b)+\sinh(s-a)\sinh(s-c)+\sinh(s-b)\sinh(s-c)\right)}{\sinh(s-a)\sinh(s-b)\sinh(s-c)}=
$$
$$
=\frac{\sinh s}{n}\left(\sinh(s-a)\sinh(s-b)+\sinh(s-a)\sinh(s-c)+\right.
$$
$$
\left.+\sinh(s-b)\sinh(s-c)\right)=
$$
$$
=\frac{\left(\sinh(s-a)\sinh(s-b)+\sinh(s-a)\sinh(s-c)+\sinh(s-b)\sinh(s-c)\right)}{\tanh r}=
$$
$$
=\frac{1}{2\tanh r}\left(\cosh a+\cosh b+\cosh c-\cosh (b-a)-\cosh (c-a)-\cosh (c-b)\right)
$$
implying (35).
From (33) and (35) we get
$$
\tanh r\left(\tanh r_A+\tanh r_B+\tanh r_C\right)+\tanh r_A\tanh r_B+\tanh r_A\tanh r_C+
$$
$$
+\tanh r_B\tanh r_C=\cosh (a+b)+\cosh (a+c)+\cosh (b+c)-\cosh (b-a)-\cosh (c-a)-
$$
$$
-\cosh (c-b)=2(\sinh a\sinh b+\sinh a\sinh c+\sinh b\sinh c)
$$
which implies (36).
\end{proof}
The following theorem gives a connection from among the distance of the incenter and circumcenter, the radiuses $r,R$ and the side-lengths $a,b,c$ .
\begin{theorem}[\cite{gho 3}]
Let $O$ and $I$ the center of the circumsrcibed and inscribed circles, respectively. Then we have
\begin{equation}
\cosh OI=2\cosh \frac{a}{2}\cosh \frac{b}{2}\cosh \frac{c}{2}\cosh r\cosh R+\cosh\frac{a+b+c}{2}\cosh(R-r).
\end{equation}
\end{theorem}

\section{Further formulas on hyperbolic triangles}

\subsection{On the orthocenter of a triangle.}

The most important formulas on the orthocenter are also valid in the hyperbolic plane. We give a collection in which the orthocenter is denoted by $H$, the feet of the altitudes are denoted by $H_A$, $H_B$ and $H_C$, respectively. We also denote by $h_a$, $h_b$ or $h_c$ the heights of the triangle corresponding to the sides $a$, $b$ or $c$, respectively.
\begin{theorem}
With the notation above we have the formulas:
\begin{equation}
\tanh HA\cdot \tanh HH_A=\tanh HB\cdot \tanh HH_B=\tanh HC\cdot \tanh HH_C=:h
\end{equation}
\begin{eqnarray}
\sinh HA\cdot \sinh HH_A \mbox{ }:\mbox{ } \sinh HB\cdot \sinh HH_B &: &\sinh HC\cdot \sinh HH_C= \\
\nonumber &=& \cosh h_A:\cosh h_B:\cosh h_C
\end{eqnarray}
\begin{equation}
n_A(H):n_B(H):n_C(H)= \tan \alpha:\tan \beta:\tan \gamma.
\end{equation}
Furthermore let $P$ be any point of the plane then we have
\begin{equation}
n_A(H)\cosh PA+n_B(H)\cosh PB+n_C(H)\cosh PC=n\cosh PH
\end{equation}
and also
\begin{equation}
\cosh c\sinh H_AC+\cosh b\sinh BH_A=\cosh h_A\sinh a.
\end{equation}
Finally we have also that
\begin{equation}
(h+1)\cosh OH=\left(\frac{\coth h_A}{\sinh HA}+\frac{\coth h_B}{\sinh HB}+\frac{\coth h_C}{\sinh HC}\right)\cosh R.
\end{equation}

\end{theorem}

\begin{figure}[htbp]
\centerline{\includegraphics[scale=1]{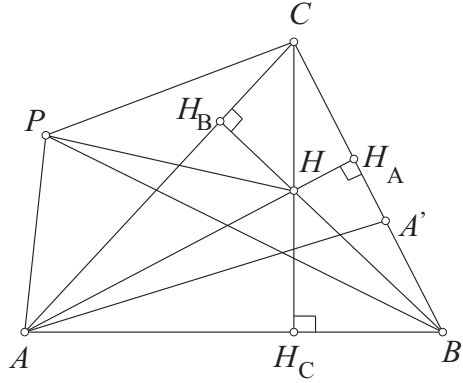}}
\caption{Stewart's theorem and the orthocenter.}
\end{figure}

Before the proof we prove Stewart's Theorem on the hyperbolic plane.
\begin{theorem}[{\bf Stewart's theorem}]
Let $ABC$ be a triangle and $A'$ is a point on the side $BC$. Then we have
\begin{equation}
\cosh AB\sinh A'C+\cosh AC\sinh BA'=\cosh AA'\sinh BC.
\end{equation}
\end{theorem}

\begin{proof}
Using (2) to the triangles $ABA'$ and $ACA'$, respectively, we get
$$
\cosh AA'\sinh BC=\cosh AA'\sinh(BA'+A'C)=\sinh BA'\cosh A'C\cosh AA'+
$$
$$
+\sinh A'C\cosh BA'\cosh AA'
=\sinh BA'(\sinh A'C\sinh AA'\cos (AA'C_\measuredangle)+
$$
$$
+\cosh AC)+\sinh A'C(\sinh BA'\sinh AA'\cos (\pi-AA'C_\measuredangle)+\cosh AB)=
$$
$$
=\sinh BA'\cosh AC+\sinh A'C\cosh AB
$$
as we stated.
\end{proof}

\begin{remark}
Considering third-order approximation of the hyperbolic functions we get the equality:
$$
\left(1+\frac{{AA'}^2}{2}\right)\left(BC+\frac{BC^3}{6}\right)=\left(1+\frac{b^2}{2}\right)\left(BA'+\frac{{BA'}^3}{6}\right)+
$$
$$
+\left(1+\frac{c^2}{2}\right)\left(A'C+\frac{{A'C}^3}{6}\right)
$$
or equivalently the equation
$$
a+\frac{{AA'}^2}{2}a+\frac{a^3}{6}=BA'+\frac{b^2}{2} BA' +\frac{{BA'}^3}{6}+A'C+\frac{c^2}{2} A'C+ \frac{{A'C}^3}{6}.
$$
Since $a=BA'+A'C$
$$
\frac{{AA'}^2}{2}a+\left(\frac{{BA'}^3}{6}+\frac{{BA'}^2 A'C}{2}+\frac{BA'A'C^2}{2}+\frac{A'C^3}{6}\right)=
$$
$$
=\frac{b^2}{2} BA' +\frac{{BA'}^3}{6}+\frac{c^2}{2} A'C+ \frac{{A'C}^3}{6}
$$
implying the well-known Euclidean Stewart's theorem:
$$
\left({AA'}^2+BA'\cdot A'C\right)a=b^2 BA'+c^2 A'C.
$$

\end{remark}

\begin{proof}(Proof of Theorem 2.1)
(43) is the Stewart's theorem for the point $H_A$.

From the rectangular triangles $HCH_A$ and $HH_CA$ we get that $\tanh HH_A : \tanh HC=\cos H_AHC\measuredangle =\tanh HH_C : \tanh HA$. Similarly we get also that $\tanh HH_B : \tanh HC=\cos H_BHC\measuredangle=\tanh HH_C : \tanh HB$ thus we have (39):
$$
\tanh HA\cdot \tanh HH_A=\tanh HB\cdot \tanh HH_B=\tanh HC\cdot \tanh HH_C.
$$
From this we get
$$
\frac{\sinh HA\cdot \sinh HH_A}{\cosh HA\cdot \cosh HH_A}=\frac{\sinh HB\cdot \sinh HH_B}{\cosh HB\cdot \cosh HH_B}.
$$
Thus
$$
\frac{\sinh HA\cdot \sinh HH_A}{\sinh HB\cdot \sinh HH_B}=\frac{\cosh HA\cdot \cosh HH_A}{\cosh HB\cdot \cosh HH_B}=\frac{\cosh AH_B}{\cosh BH_A}
$$
implying (40). From (9) we get that
$$
n_A(H):n_B(H)=(AH_CB)=\sinh AH_C: \sinh H_CB=\tan \alpha:\tan\beta
$$
implying (41).
Use now the Stewart's Theorem for the triangle $PAB$ and its secant $PH_C$ (see in Fig. 4), where $P$ is arbitrary point of the plane. Then we get
$$
\cosh PA \sinh H_CB+\cosh PB \sinh AH_C=\cosh PH_C \sinh c.
$$
Applying Stewart's theorem again to the triangle $PCH_C$ and its secant $PH$, we get
$$
\cosh PC \sinh HH_C+\cosh PH_C \sinh CH=\cosh PH \sinh CH_C.
$$
Eliminating $PH_C$ from these equations we get
$$
\cosh PA \sinh H_CB+\cosh PB \sinh AH_C+\frac{\cosh PC \sinh HH_C\sinh c}{\sinh CH}=
$$
$$
=\frac{\cosh PH \sinh CH_C\sinh c}{\sinh CH}.
$$
On the other hand we have
$$
2n_C(H)=\sinh HH_C \sinh c.
$$
We also have
$$
2n_B(H)=2\sinh HH_B \sinh b=2\sinh CH_A\sinh AH=2\sinh AH_C \sinh CH,
$$
and similarly
$$
2n_A(H)=2\sinh H_CB \sinh CH
$$
implying the equality
$$
n_A(H)\cosh PA+n_B(H)\cosh PB+n_C(H)\cosh PC=\frac{\cosh PH \sinh CH_C\sinh c}{2}=
$$
$$
=n\cosh PH
$$
as we stated in (42).

Use (42) in the case when $P=O$ is the circumcenter of the triangle. Then we have
\begin{equation}
n_A(H)\cosh R+n_B(H)\cosh R+n_C(H)\cosh R=n\cosh OH.
\end{equation}
Thus we have
$$
\cosh OH=\frac{n_A(H)+n_B(H)+n_C(H)}{n}\cosh R=
$$
$$
=\left(\frac{\sinh HH_A}{\sinh h_A}+\frac{\sinh HH_B}{\sinh h_B}+\frac{\sinh HH_C}{\sinh h_C}\right)\cosh R.
$$
From (40) we get
$$
\sinh HH_B=\sinh HH_A\frac{\sinh HA}{\sinh HB}\frac{\cosh h_B}{\cosh h_A}
$$
and also
$$
\sinh HH_C=\sinh HH_A\frac{\sinh HA}{\sinh HC}\frac{\cosh h_C}{\cosh h_A}
$$
implying that
$$
\cosh OH=\frac{\sinh HH_A\sinh HA}{\cosh h_A}\times
$$
$$
\times\left(\frac{\cosh h_A}{\sinh HA\sinh h_A}+\frac{\cosh h_B}{\sinh HB\sinh h_B}+\frac{\cosh h_C}{\sinh HC\sinh h_C}\right)\cosh R=
$$
$$
=\left(\frac{\cosh h_A}{\sinh HA\sinh h_A}+\frac{\cosh h_B}{\sinh HB\sinh h_B}+\frac{\cosh h_C}{\sinh HC\sinh h_C}\right)\times
$$
$$
\times\frac{\cosh R}{\tanh HH_A\tanh HA+1}.
$$
Now we have
$$
(h+1)\cosh OH=
$$
$$
=\left(\frac{1}{\tanh h_A\sinh HA}+\frac{1}{\tanh h_B\sinh HB}+\frac{1}{\tanh h_C\sinh HC}\right)\cosh R,
$$
showing (44).

\end{proof}

\subsection{Isogonal conjugate of a point}

Let define the \emph{isogonal conjugate} of a point $X$ of the plane in the following way: Reflect the lines through the point $X$ and any of the vertices of the triangle with respect to the bisector of that vertex. Then the getting lines are concurrent at a point $X'$ which we call the isogonal conjugate of $X$. To prove the concurrence of these lines we have to observe that if the lines $AX$ and $AX'$ intersect the line of the side $BC$ in the points $Y$ and $Y'$ then the ratio of these points with respect to $B$ and $C$ has an inverse connection. In fact, by (1) we have that
$$
\frac{\sinh c}{\sinh BY}=\frac{\sin AYB\measuredangle}{\sin BAY\measuredangle} \quad \mbox{ and } \quad \frac{\sinh b}{\sinh YC}=\frac{\sin (\pi- AYB\measuredangle)}{\sin CAY\measuredangle}.
$$
This implies that
$$
(BYC)=\frac{\sinh BY}{\sinh YC}=\frac{\sinh c}{\sinh b}\frac{\sin BAY\measuredangle}{\sin CAY\measuredangle}.
$$
For the point $Y'$ we get similarly that
$$
(BY'C)=\frac{\sinh c}{\sinh b}\frac{\sin BAY'\measuredangle}{\sin CAY'\measuredangle}=\frac{\sinh c}{\sinh b}\frac{\sin CAY\measuredangle}{\sin BAY \measuredangle}
$$
implying the equation
\begin{equation}
(BYC)(BY'C)=\frac{\sinh ^2 c}{\sinh ^2 b}.
\end{equation}
If $Z,Z'$ or $V,V'$ are the intersection points of the examined lines with the corresponding sides $CA$ or $AB$, respectively, then we get the equation
$$
(BYC)(BY'C)(CZA)(CZ'A)(AVB)(AV'B)=1
$$
showing that the first three lines are concurrent if and only if the second three lines are. Hence we can prove the following:
\begin{lemma}
If $X$ and $X'$ are isogonal conjugate points with respect to the triangle $ABC$ then their triangular coordinates have the following connection:
\begin{equation}
n_A(X'):n_B(X'):n_C(X')=\frac{\sinh^2a}{n_A(X)}:\frac{\sinh^2b}{n_B(X)}:\frac{\sinh ^2c}{n_C(X)}.
\end{equation}
\end{lemma}

\begin{proof}
Using (47) we have
$$
\left(n_C(X):n_B(X)\right)\left(n_C(X'):n_B(X')\right)=(BN_AC)(BN'_AC)=\frac{\sinh ^2 c}{\sinh ^2 b}
$$
implying that
$$
n_B(X'):n_C(X')=\frac{\sinh^2b}{n_B(X)}:\frac{\sinh ^2c}{n_C(X)}
$$
as we stated in (48).
\end{proof}

\begin{corollary}
As a first consequence we can see immediately (30) again on the triangular coordinates of the incenter. By (48) the triangular coordinates of the isogonal conjugate $H'$ of the orthocenter is
$$
n_A(H'):n_B(H'):n_C(H')=\frac{\sinh^2a}{\tan\alpha}:\frac{\sinh^2b}{\tan\beta}:\frac{\sinh ^2c}{\tan\gamma}.
$$
Thus
$$
n_A(H'):n_B(H')=\frac{\sinh^2a}{\tan\alpha}\frac{\tan\beta}{\sinh^2b}=\frac{\sin \alpha\cos\alpha}{\sin \beta\cos\beta}=\frac{\sin 2\alpha}{\sin 2\beta}
$$
implying that
\begin{equation}
n_A(H'):n_B(H'):n_C(H')=\sin 2\alpha:\sin 2\beta:\sin 2\gamma.
\end{equation}
Comparing the coordinates of $H'$ with the triangular coordinates of the circumcenter we can see that the isogonal conjugate of the orthocenter is the circumcenter if and only if the defect of the triangle is zero implying that the geometry of the plane is Euclidean.
\end{corollary}

A minimality property of the incenter follows from a generalization of the equality (42). Similarly as in the proof of (42) (see Theorem 2.1) we can prove that for any triangle $ABC$ with any fixed point $Q$ and any various point $P$ of the plane the following equality holds:
\begin{equation}
n_A(Q)\cosh PA+n_B(Q)\cosh PB+n_C(Q)\cosh PC=n(ABC)\cosh PQ.
\end{equation}

\begin{theorem}
The sum of the triangular coordinates of a point $P$ of the plane is minimal if and only if $P$ is the center of the inscribed circle of the triangle $ABC$.
\end{theorem}

\begin{proof}
Assume that the vertices of the triangle $ABC$ are real points and the edges of it are those real segments which are connecting these real vertices, respectively.
Let $A',B'$ and $C'$ be the respective poles of the lines $BC$, $AC$ and $AB$. These poles are ideal points and the corresponding lines $A'B'$, $A'C'$ and $B'C'$ are also ideal lines, respectively. If $P$ is any point of the plane let $d(P, BC)$, $\varepsilon_A$ and $\alpha'$ be the distance of $P$ and the line $BC$ the sign of this distance and the angle of the polar triangle at the vertex $A'$, respectively. We choose the sign to positive if $P$ and $A$ are the same (real) half-plane determined by the line $BC$. Then the investigated quantity is
$$
n_A(P)+n_B(P)+n_C(P)=
$$
$$
=\frac{1}{2}\left(\varepsilon_A\sinh d(P, BC)\sinh a+\varepsilon_B\sinh d(P, AC)\sinh b+\right.
$$
$$
\left.+\varepsilon_C\sinh d(P, AB)\sinh c\right)=
$$
$$
=\frac{1}{2i}\left(\cosh \left(d(P, BC)+\varepsilon_A\frac{\pi}{2}i\right)\sinh a +\cosh \left(d(P, AC)+\varepsilon_B\frac{\pi}{2}i\right)\sinh b+\right.
$$
$$
\left. +\cosh \left(d(P, AB)+\varepsilon_C\frac{\pi}{2}i\right)\sinh c\right)=
$$
$$
=\frac{1}{2i}\left(\cosh PA' \sinh a +\cosh PB' \sinh b+\cosh PC' \sinh c\right).
$$
Hence using (50) we have that
$$
\frac{1}{2i}\left(\cosh PA' \sinh a +\cosh PB' \sinh b+\cosh PC' \sinh c\right)=\frac{1}{2i}n(A'B'C')\cosh PQ
$$
where the triangular coordinates of the point $Q$ with respect to the polar triangle are
$$
n_{A'}(Q)=\sinh a, \quad n_{B'}(Q)=\sinh b, \quad \mbox{ and } \quad n_{C'}(Q)=\sinh c.
$$
It follows from (8) that the Staudtian  of the triangle $A'B'C'$ is
$$
n(A'B'C')=\frac{1}{2}\sin \alpha'\sinh b'\sinh c'=\frac{1}{2}\sin\frac{a}{i}\sinh i\beta\sinh i\gamma=\frac{i}{2}\sinh a\sin \beta\sin\gamma
$$
implying that.
$$
n_A(P)+n_B(P)+n_C(P)=\frac{1}{4}\sinh a\sin \beta\sin\gamma \cosh PQ=\frac{N}{2}\cosh PQ,
$$
where the triangular coordinates of $Q$ are $\sinh a$, $\sinh b$ and $\sinh c$, respectively. Thus we get that $Q=I$ and the sum in the question is minimal if and only if $P$ is equal to $Q=I$. This proves the statement.
\end{proof}

\subsubsection{Symmedian point}
We recall that the isogonal conjugate of the centroid is the so-called \emph{symmedian point} of the triangle. The triangular coordinates of the symmedian point are
\begin{equation}
n_A(M'):n_B(M'):n_C(M')=\sinh^2a :\sinh^2b:\sinh ^2c.
\end{equation}
From (8) immediately follows that the hyperbolic sine of the distances of the symmedian point to the sides are proportional to the hyperbolic sines of the corresponding sides:
\begin{equation}
\sinh d(M',BC):\sinh d(M',AC):\sinh d(M',AB)=\sinh a:\sinh b:\sinh c
\end{equation}
showing the validity of the analogous Euclidean theorem in the hyperbolic geometry, too.

\begin{figure}[htbp]
\centerline{\includegraphics[scale=1]{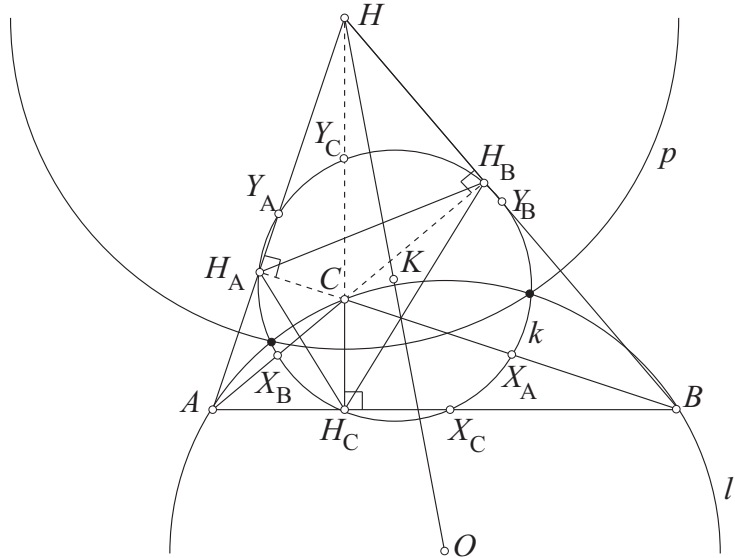}}
\caption{The Lemoine point of the triangle.}
\end{figure}

We note that the symmedian point of a hyperbolic triangle does not coincides with the \emph{Lemoine point} $L$ of the triangle. This center can be defined on the following way: If tangents be drawn at $A,B,C$ to the circumcircle of the triangle $ABC$, forming a triangle $A'B'C'$, the lines $AA'$, $BB'$ and $CC'$, are concurrent. The point of concurrence, is the Lemoine point of the triangle. The concurrency follows from Menelaos-theorem applying it to the triangle $A'B'C'$. We note that $L$ is also (by definition) the so-called \emph{Gergonne point} of the triangle $A'B'C'$.
To prove that the symmedian point does not coincides with the Lemoine point we determine the triangular coordinates of the latter, too. Let $L_A$, $L_B$ or $L_C$ be the intersection point of $AA'\cap BC$, $BB'\cap AC$ or $CC'\cap AB$ (see in Fig. 5),respectively. Then we have
$$
n_B(L):n_A(L)=(AL_CB)=\frac{\sinh AL_C}{\sinh L_CB}=\frac{\sinh C'B}{\sinh C'A}\frac{\sin AC'L_C\measuredangle}{\sin BC'L_C\measuredangle}=\frac{\sin AC'L_C\measuredangle}{\sin BC'L_C\measuredangle}.
$$
On the other hand we have by (1)
$$
\frac{\sin AC'L_C\measuredangle}{\sin CAC'\measuredangle}=\frac{\sinh CA}{\sinh CC'} \quad \mbox{ and }\quad \frac{\sin BC'L_C\measuredangle}{\sin CBC'\measuredangle}=\frac{\sinh CB}{\sinh CC'}
$$
implying that
$$
\frac{\sin AC'L_C\measuredangle}{\sin BC'L_C\measuredangle}=\frac{\sinh CA\sin CAC'\measuredangle}{\sinh CB\sin CBC'\measuredangle}=\frac{\sinh CA\cos CAO\measuredangle}{\sinh CB\sin CBO\measuredangle}=
$$
$$
=\frac{ 2\sinh \frac{CA}{2}\cosh \frac{CA}{2}\cos CAO\measuredangle}{2\sinh \frac{CB}{2}\cosh \frac{CB}{2}\sin CBO\measuredangle }=
$$
$$
=\frac{\sinh \frac{CA}{2}\cosh \frac{CA}{2}\frac{\tanh \frac{CA}{2}}{\tanh R}}{\sinh \frac{CB}{2}\cosh \frac{CB}{2}\frac{\tanh \frac{CB}{2}}{\tanh R}}=\frac{\sinh^2 \frac{b}{2}}{\sinh ^2\frac{a}{2}}=\left(\cosh b-1\right):\left(\cosh a -1\right).
$$
Thus the triangular coordinates of the Lemoine point are:
\begin{equation}
n_A(L):n_B(L):n_C(L)=\left(\cosh a -1\right):\left(\cosh b -1\right):\left(\cosh a -1\right).
\end{equation}
Now the symmedian point and the Lemoine point coincides for a triangle if and only if the equation array
\begin{eqnarray}
\left(\cosh a -1\right)\sinh^2 b & = & \left(\cosh b -1\right)\sinh ^2 a \\
\nonumber \left(\cosh a -1\right)\sinh^2 c & = & \left(\cosh c -1\right)\sinh ^2 a
\end{eqnarray}
gives an identity.
Since
$$
(\cosh a -1)(\cosh ^2 b-1)=(\cosh a -1)(\cosh b-1)(\cosh b+1)=
$$
$$
=(\cosh b-1)(\cosh a -1)(\cosh a +1)=
$$
$$
=\left(\cosh b -1\right)\sinh ^2 a
$$
implies $a=b$, the only solution is when $a=b=c$ and the triangle is an equilateral (regular) one.

\subsection{On the ``Euler line''.}

An interesting question in elementary hyperbolic geometry is the existence of the Euler line. Known fact (see e.g. in \cite{szasz}) that the circumcenter, the centroid and the orthocenter of a triangle having in a common line if and only if the triangle is isoscale. In this sense Euler line does not exist for each triangle. A nice result from the recent investigations on the triangle centers is the paper of A.V. Akopyan \cite{akopyan} in which the author defined the concepts of ''pseudomedians'' and ''pseudoaltitudes" giving two new centers of the hyperbolic triangle holding a deterministic Euclidean property of Euclidean centroid and orthocenter, respectively. He proved that the circumcenter, the intersection points of the pseudomedians (pseudo-centroid), the intersection points of the pseudoaltitudes (pseudo-orthocenter) and the circumcenter of the circle through the footpoints of the bisectors (the center of the Feuerbach circle) are on a hyperbolic line.
A line through a vertex is called by \emph{pseudomedian} if divides the area of the triangle in half. (We note that in spherical geometry Steiner proved the statement that the great circles through angular points of a spherical triangle, and which bisect its area, are concurrent (see \cite{casey 1}). Of course the pseudomedians are not medians and their point of concurrency is not the centroid of the triangle. We call it \emph{pseudo-centroid}. He called \emph{pseudoaltitude} a cevian ($AZ_A$) with the property that with its foot $Z_A$ on $BC$ holds the equality $$
AZ_AB\measuredangle-Z_ABA\measuredangle-BAZ_A\measuredangle=CZ_AA\measuredangle-Z_AAC\measuredangle-ACZ_a\measuredangle
$$
where the angles above are directed, respectively.
Throughout on his paper Ako\-py\-an assume that {\it ``any two lines intersects and that three points determine a circle''}. He note in the introduction also that {\it ``Consideration of all possible cases
would not only complicate the proof, but would contain no fundamentally new ideas. To complete
our arguments, we could always say that other cases follow from a theorem by analytic continuation,
since the cases considered by us are sufficiently general (they include an interior point in the
configuration space).
Nevertheless, in the course of our argument we shall try to avoid major errors and show that
the statements can be demonstrated without resorting to more powerful tools''}. We note that in our paper the reader can find this required extraction of the real elements by the ideal elements and the elements at infinity. We also defined all concepts using by Akopyan with respect to general points and lines, furthermore his lemmas and theorem can be extracted from circles onto cycles with our method. This prove the truth of Akopyan's note, post factum.
\begin{figure}[htbp]
\centerline{\includegraphics[scale=1]{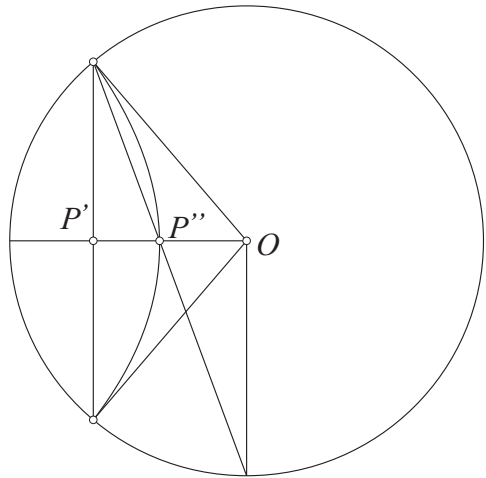}}
\caption{The connection between the projective and conformal models}
\end{figure}
To see the equivalence of the two theory on real elements we recall that between the projective (Cayley-Klein-Beltrami) and Poincare models of the unit disk there is a natural correspondence, when we map to a line of the projective model to the line of the Poincare model with the same ends (points at infinity). On Fig. 6 we can see the corresponding mapping. A point $P$ can be realized in the first model as the point $P'$ and in the second one as the point $P''$. It is easy to see that if the hyperbolic distance of the points $P$ and $O$ is $a$ then the Euclidean distances $P'O$ or $P''O$ are equals to $\tanh a$ or $\tanh(a/2)$, respectively.
Thus our analytic definitions on similarity or inversion are model independent (end extracted ) variations of the definitions of Akopyan, respectively. Thus we have

\begin{theorem}[\cite{akopyan}]
The center $O$ of the cycle around the triangle, the center of the cycle $F$ around the feet of the pseudomedians, the pseudo-centroid $S$ and the pseudo-orthocenter $Z$ are on the same line.
\end{theorem}

By Akopyan's opinion this is the \emph{Euler line} of the triangle and thus he avoided the problem is to determination of the connection among the three important classical centers of the triangle. Our aim to give some analytic determination for the pseudo-centers introduced by Akopyan.

\begin{theorem}Let $S_A,S_B,S_C$ be the feet of the pseudo-medians. Then we have the following formulas:
\begin{eqnarray}
\sinh \frac{AN_C}{2}:\sinh\frac{N_CB}{2} & = & \cosh \frac{b}{2}:\cosh \frac{a}{2} \\
\nonumber \sinh \frac{B N_A}{2}:\sinh\frac{N_AC}{2} & = & \cosh \frac{c}{2}:\cosh \frac{b}{2} \\
\nonumber \sinh \frac{C N_B}{2}:\sinh\frac{N_BA}{2} & = & \cosh \frac{a}{2}:\cosh \frac{c}{2}
\end{eqnarray}
implying that they are concurrent in a point $S$. We call $S$ the \emph{pseudo-centroid} of the triangle.
The triangular coordinates of the pseudo-centroid hold:
\begin{eqnarray}
n_A(R):n_B(R):n_C(R)&  = &\frac{1}{\left(\cosh^2 \frac{b}{2}\cosh ^2\frac{c}{2} + \cosh \frac{a}{2}\cosh \frac{b}{2}\cosh \frac{c}{2}\right)}:\\
\nonumber & : &\frac{1}{\left(\cosh ^2 \frac{a}{2}\cosh^2 \frac{c}{2} + \cosh \frac{a}{2}\cosh \frac{b}{2}\cosh \frac{c}{2}\right)}:\\
\nonumber & : & \frac{1}{\left(\cosh ^2 \frac{b}{2}\cosh ^2\frac{a}{2}+\cosh \frac{a}{2}\cosh \frac{b}{2}\cosh \frac{c}{2}\right)}.
\end{eqnarray}
\end{theorem}

\begin{proof}
We know that
$$
\cosh \frac{a}{2}\cosh \frac{b}{2}\cosh \frac{c}{2}=\frac{N^2}{\sin \alpha\sin\beta\sin \gamma\sin\delta}.
$$
(15) says that
$$
2n^2=N\sinh a\sinh b\sinh c,
$$
and we also have
$$
\sin \alpha\sin\beta\sin \gamma\sinh a\sinh b\sinh c=4nN.
$$
From these equalities we get the analogous of the spherical Cagnoli's theorem:
\begin{eqnarray}
\sin\delta =\frac{N^2}{\sin \alpha\sin\beta\sin \gamma\cosh \frac{a}{2}\cosh \frac{b}{2}\cosh \frac{c}{2}} & = & \\
\nonumber =\frac{N^2\sinh a\sinh b\sinh c}{4nN\cosh \frac{a}{2}\cosh \frac{b}{2}\cosh \frac{c}{2}} & = &\frac{n}{2\cosh \frac{a}{2}\cosh \frac{b}{2}\cosh \frac{c}{2}}.
\end{eqnarray}
But
$$
\cosh \frac{a}{2}\sinh\frac{b}{2}\sinh\frac{c}{2}=
$$
$$
=\sqrt{\frac{\sin{(\delta+\beta)}\sin{(\delta+\gamma})}{\sin \gamma\sin\beta}}\sqrt{\frac{\sin{\delta}\sin{(\delta+\beta})}{\sin \gamma\sin\alpha}}\sqrt{\frac{\sin{\delta}\sin{(\delta+\gamma})}{\sin \alpha\sin\beta}}=
$$
$$
=\frac{N^2}{\sin(\delta+\alpha)\sin\alpha\sin\beta\sin \gamma},
$$
implying (with the above manner) the equality
\begin{equation}
\sin(\delta+\alpha)=\frac{n}{2\cosh \frac{a}{2}\sinh \frac{b}{2}\sinh \frac{c}{2}}.
\end{equation}
From these equalities we get that
\begin{equation}
\frac{\sin(\delta+\alpha)}{\sin\delta}=\cos\alpha +\cot \delta\sin\alpha=\coth \frac{b}{2}\coth \frac{c}{2}.
\end{equation}
Thus if the area of a triangle and one of its angles be given, the product of the semi hyperbolic tangents of the containing sides is given.
Since the area of the examined triangles are equals to each other we get that
$$
\frac{n}{2\cosh \frac{a}{2}\cosh \frac{BN_C}{2}\cosh \frac{CN_C}{2}}=\frac{\sinh a\sinh BN_C\sin \beta}{4\cosh \frac{a}{2}\cosh \frac{BN_C}{2}\cosh \frac{CN_C}{2}}=
$$
$$
=\frac{\sinh \frac{a}{2}\sinh \frac{BN_C}{2}\sin \beta}{\cosh \frac{CN_C}{2}}
$$
and similarly
$$
\frac{n}{2\cosh \frac{b}{2}\cosh \frac{N_CA}{2}\cosh \frac{CN_C}{2}}=\frac{\sinh \frac{b}{2}\sinh \frac{N_CA}{2}\sin \alpha}{\cosh \frac{CN_C}{2}}
$$
implying that
$$
\sinh \frac{a}{2}\sinh \frac{BN_C}{2}\sin \beta=\sinh \frac{b}{2}\sinh \frac{N_CA}{2}\sin \alpha.
$$
From this we get that
$$
\frac{\sinh \frac{AN_C}{2}}{\sinh \frac{N_CB}{2}}=\frac{\sinh \frac{a}{2}\sin \beta}{\sinh \frac{b}{2}\sin \alpha}=\frac{\cosh \frac{b}{2}}{\cosh \frac{a}{2}}
$$
as we stated in (55). The product of the equalities in (55) gives the equality
\begin{equation}
\sinh \frac{AN_C}{2}\sinh \frac{B N_A}{2}\sinh \frac{C N_B}{2}=\sinh\frac{N_CB}{2}\sinh\frac{N_AC}{2}\sinh\frac{N_BA}{2}.
\end{equation}
On the other hand the triangles $CAN_C$, $N_BAB$ having equal areas and also have a common angle, in virtue of
(67) we get that
$$
\tanh \frac{b}{2}\tanh \frac{AN_C}{2}=\tanh \frac{c}{2}\tanh \frac{N_BC}{2},
$$
implying that
$$
\tanh \frac{AN_C}{2}\tanh \frac{BN_A}{2}\tanh \frac{CN_B}{2}=\tanh \frac{N_BC}{2}\tanh \frac{N_CB}{2}\tanh \frac{N_AC}{2}.
$$
So we also have
$$
\cosh \frac{AN_C}{2}\cosh \frac{BN_A}{2}\cosh \frac{CN_B}{2}=\cosh \frac{N_BC}{2}\cosh \frac{N_CB}{2}\cosh \frac{N_AC}{2},
$$
and as a consequence the equality
$$
\sinh AN_C\sinh B N_A\sinh C N_B=\sinh N_CB\sinh N_AC \sinh N_BA.
$$
Menelaos theorem now gives the existence of the pseudo-centroid.

From (55) we get that
$$
\frac{\cosh \frac{a}{2}}{\cosh \frac{b}{2}}=\frac{\sinh \left(\frac{c}{2}-\frac{AN_C}{2}\right)}{\sinh \frac{AN_C}{2}}=\sinh \frac{c}{2}\coth \frac{AN_C}{2}-\cosh \frac{c}{2},
$$
hence
$$
\coth \frac{AN_C}{2}=\frac{\cosh \frac{b}{2}\cosh \frac{c}{2}+\cosh \frac{a}{2}}{\sinh \frac{c}{2}\cosh \frac{b}{2}}
$$
or equivalently
$$
\cosh \frac{AN_C}{2}=\frac{\cosh \frac{b}{2}\cosh \frac{c}{2}+\cosh \frac{a}{2}}{\sinh \frac{c}{2}\cosh \frac{b}{2}}\sinh \frac{AN_C}{2}.
$$
From this we get
$$
1=\sinh^2\frac{AN_C}{2}\left(-1+\left(\frac{\cosh \frac{b}{2}\cosh \frac{c}{2}+\cosh \frac{a}{2}}{\sinh \frac{c}{2}\cosh \frac{b}{2}}\right)^2\right)=
$$
$$
=\frac{-\sinh ^2\frac{c}{2}\cosh ^2\frac{b}{2}+\left(\cosh \frac{b}{2}\cosh \frac{c}{2}+\cosh \frac{a}{2}\right)^2}{\sinh ^2\frac{c}{2}\cosh ^2\frac{b}{2}}\sinh^2\frac{AN_C}{2}=
$$
$$
\frac{\cosh ^2\frac{b}{2}+2\cosh \frac{a}{2}\cosh\frac{b}{2}\cosh \frac{c}{2}+\cosh ^2\frac{a}{2}}{\sinh ^2\frac{c}{2}\cosh ^2\frac{b}{2}}\sinh^2\frac{AN_C}{2}
$$
Thus
$$
\sinh AN_C=2\sinh \frac{AN_C}{2}\cosh \frac{AN_C}{2}=2\sinh ^2\frac{AN_C}{2}\frac{\cosh \frac{b}{2}\cosh \frac{c}{2}+\cosh \frac{a}{2}}{\sinh \frac{c}{2}\cosh \frac{b}{2}}=
$$
$$
=2\frac{\sinh \frac{c}{2}\cosh \frac{b}{2}\left(\cosh \frac{b}{2}\cosh \frac{c}{2}+\cosh \frac{a}{2}\right)}{\cosh ^2\frac{b}{2}+2\cosh \frac{a}{2}\cosh\frac{b}{2}\cosh \frac{c}{2}+\cosh ^2\frac{a}{2}}.
$$
Hence we also have
$$
\sinh N_CB=2\frac{\sinh \frac{c}{2}\cosh \frac{a}{2}\left(\cosh \frac{a}{2}\cosh \frac{c}{2}+\cosh \frac{b}{2}\right)}{\cosh ^2\frac{a}{2}+2\cosh \frac{a}{2}\cosh\frac{b}{2}\cosh \frac{c}{2}+\cosh ^2\frac{b}{2}}
$$
implying that
$$
n_B(N):n_A(N)=(AN_CB)=
$$
$$
=\left(\cosh^2 \frac{b}{2}\cosh \frac{c}{2}+\cosh \frac{b}{2}\cosh \frac{a}{2}\right):\left(\cosh ^2 \frac{a}{2}\cosh \frac{c}{2}+\cosh \frac{a}{2}\cosh \frac{b}{2}\right)=
$$
$$
=\left(\cosh^2 \frac{b}{2}\cosh ^2\frac{c}{2}+\cosh \frac{a}{2}\cosh \frac{b}{2}\cosh \frac{c}{2}\right):
$$
$$
:\left(\cosh ^2 \frac{a}{2}\cosh^2 \frac{c}{2}+\cosh \frac{a}{2}\cosh \frac{b}{2}\cosh \frac{c}{2}\right).
$$
From this we get that
$$
n_A(N):n_B(N)=\frac{1}{\left(\cosh^2 \frac{b}{2}\cosh ^2\frac{c}{2}+\cosh \frac{a}{2}\cosh \frac{b}{2}\cosh \frac{c}{2}\right)}:
$$
$$
:\frac{1}{\left(\cosh ^2 \frac{a}{2}\cosh^2 \frac{c}{2}+\cosh \frac{a}{2}\cosh \frac{b}{2}\cosh \frac{c}{2}\right)}.
$$
Similarly we get
$$
n_B(N):n_C(N)=\frac{1}{\left(\cosh^2 \frac{c}{2}\cosh ^2\frac{a}{2}+\cosh \frac{c}{2}\cosh \frac{b}{2}\cosh \frac{a}{2}\right)}:
$$
$$
:\frac{1}{\left(\cosh ^2 \frac{b}{2}\cosh ^2\frac{a}{2}+\cosh \frac{a}{2}\cosh \frac{b}{2}\cosh \frac{c}{2}\right)}
$$
as we stated in (56).
\end{proof}

\begin{remark}
We note that there are many Euclidean theorems can be investigated on the hyperbolic plane by our more-less trigonometric way. We note that on the hyperbolic plane the usual isoptic property of the circle lost (see \cite{csima}) and thus all the Euclidean statements using this property can be investigated only the way of \cite{akopyan}. To that we can use trigonometry in this method we can concentrate on the introduced concept of angle sums which in a trigonometric calculation can be handed well. Thus the isoptic property of a cycle (or which is the same the cyclical property of a set of points) can lead for new hyperbolic theorems suggested by known Euclidean analogy.
\end{remark}

\end{document}